\documentclass[12pt]{amsart}
\font\bbbld=msbm10 scaled\magstephalf

\newcommand{\bfH}{\hbox{\bbbld H}}
\newcommand{\bfR}{\hbox{\bbbld R}}
\newcommand{\bfS}{\hbox{\bbbld S}}
\newcommand{\ve}{{\bf e}}
\newcommand{\vs}{{\bf s}}

\newcommand{\vz}{{\bf z}}
\newcommand{\e}{\varepsilon}
\newcommand{\goto}{\rightarrow}

\newcommand{\ol}{\overline}

\newcommand{\be}{\begin{equation}}
\newcommand{\ee}{\end{equation}}

\newtheorem{theorem}{Theorem}[section]
\newtheorem{lemma}[theorem]{Lemma}
\newtheorem{proposition}[theorem]{Proposition}
\newtheorem{corollary}[theorem]{Corollary}

\theoremstyle{definition}

\theoremstyle{remark}

\numberwithin{equation}{section}

\begin{document}
\setlength{\baselineskip}{1.2\baselineskip}

\title[Hypersurfaces of Constant Curvature]
{
Hypersurfaces of Constant Curvature in Hyperbolic Space II.}
\author{Bo Guan}
\address{Department of Mathematics, Ohio State University,
         Columbus, OH 43210}
\email{guan@math.osu.edu}
\author{Joel Spruck}
\address{Department of Mathematics, Johns Hopkins University,
 Baltimore, MD 21218}
\email{js@math.jhu.edu}
\thanks{Research of both authors was supported in part
by NSF grants.}



\maketitle

\section{Introduction}
\label{sec1}
\setcounter{equation}{0}

In this paper we continue our study  of   complete hypersurfaces in
hyperbolic space $\bfH^{n+1}$ of constant curvature  with a
prescribed asymptotic boundary at infinity. Given $\Gamma \subset
\partial_{\infty} \bfH^{n+1}$  and a smooth symmetric function $f$ of
$n$ variables, we seek a complete hypersurface $\Sigma$ in
$\bfH^{n+1}$ satisfying
\begin{equation}
\label{eq1.10}
f(\kappa[\Sigma]) = \sigma
\ee
\be \label{eq1.20}
\partial \Sigma = \Gamma
\end{equation}
where $\kappa[\Sigma] = (\kappa_1, \dots, \kappa_n)$
denotes the hyperbolic principal curvatures of $\Sigma$ and $\sigma \in (0,1)$ 
is a constant.

We will use the half-space model,
\[ \bfH^{n+1} = \{(x, x_{n+1}) \in \bfR^{n+1}: x_{n+1} > 0\} \]
equipped with the hyperbolic metric
\begin{equation}
 ds^2 = \frac{1}{x_{n+1}^2} \sum_{i=1}^{n+1} dx_i^2\,.
\end{equation}
Thus $\partial_\infty \bfH^{n+1}$ is naturally identified with
$\bfR^n = \bfR^n \times \{0\} \subset \bfR^{n+1}$ and (\ref{eq1.20}) may
be understood in the Euclidean sense.

As in in our earlier work \cite{RS94, NS96, GS00, GSZ08}, we will take
 $\Gamma=\partial \Omega$ where
$\Omega \subset \bfR^n$ is a smooth domain and seek $\Sigma$ as  the
graph of  a function $u(x)$ over $\Omega$, i.e.
  \[\Sigma=\{(x,x_{n+1}): x\in \Omega,~ x_{n+1}=u(x)\}.\]
 Then  the coordinate vector fields and upper unit normal are  given by
 \[X_i=e_i+u_i e_{n+1},~{\bf n}=u\nu=u\frac{(-u_i e_i+e_{n+1})}{w},\]
 where $ w=\sqrt{1+|\nabla u|^2}$.
 The first fundamental form $g_{ij}$ is then given by
  \begin{equation}
\label{eq1}
 g_{ij} = \langle X_i,X_j \rangle
        = \frac1{u^2}(\delta_{ij} + u_i u_j)=\frac{g^e_{ij}}{u^2}~.
\end{equation}
To compute the second fundamental form $h_{ij}$ we use
\be \Gamma_{ij}^k=\frac1{x_{n+1}}\{-\delta_{jk}\delta_{i n+1}-\delta_{ik}\delta_{j n+1}+\delta_{ij}\delta_{k n+1}\}
\ee
to obtain
\be
\nabla_{X_i}X_j=(\frac{\delta_{ij}}{x_{n+1}}+u_{ij}-\frac{u_i u_j}{x_{n+1}})e_{n+1}-\frac{u_j e_i+u_i e_j}{x_{n+1}}~. 
\ee Then
\begin{equation}
\label{eq1.0}
\begin{aligned}
 h_{ij}& =  \, \langle \nabla_{X_i}X_j,u\nu \rangle =\frac1{uw}(\frac{\delta_{ij}}u+ u_{ij}-\frac{u_i u_j}u+2\frac{u_i u_j}u)\\
       & =  \, \frac1{u^2 w}{(\delta_{ij}+u_i u_j +u u_{ij})} 
        =  \, \frac{h^e_{ij}}{u}+\frac{g^e_{ij}}{u^2 w}.
 \end{aligned}
 \end{equation}
The hyperbolic principal curvatures $\kappa_i$ of $\Sigma$ are the
roots of the characteristic equation
\[\det(h_{ij}-\kappa g_{ij})
    =u^{-n}\det(h^e_{ij}-\frac1u(\kappa-\frac1w)g^e_{ij})=0.\]
Therefore,
 \be
 \label{eq1.30}
 \kappa_i=u\kappa^e_i +\frac{1}{w}.
 \ee
We will present other more explicit and useful expressions for the
$\kappa_i$ in Section \ref{sec2}.

The function $f$ is assumed to satisfy the fundamental structure
conditions:
\begin{equation}
\label{eq1.40}
f_i (\lambda) \equiv \frac{\partial f (\lambda)}{\partial \lambda_i} > 0
  \;\; \mbox{in $K$}, \;\; 1 \leq i \leq n,
\end{equation}
\begin{equation}
\label{eq1.50}
\mbox{$f$ is a concave function in $K$},
\end{equation}
and
\begin{equation}
\label{eq1.60}
 f > 0 \;\;\mbox{in $K$},
  \;\; f = 0 \;\;\mbox{on $\partial K$}
\end{equation}
where $K \subset \bfR^n$ is an open symmetric convex cone such that
\begin{equation}
 K^+_n := \big\{\lambda \in \bfR^n:
   \mbox{each component $\lambda_i > 0$}\big\} \subset K.
 \end{equation}
 In addition, we shall assume that $f$ is normalized
\begin{equation}
\label{eq1.70}
f(1, \dots, 1) = 1
\end{equation}
and
\begin{equation}
\label{eq1.80}
 \mbox{$f$ is homogeneous of degree one.}
\end{equation}

Since $f$ is symmetric, by (\ref{eq1.50}),
(\ref{eq1.70}) and (\ref{eq1.80}) we have
\begin{equation}
\label{eq1.90}
f (\lambda) \leq f ({\bf 1}) + \sum f_i ({\bf 1}) (\lambda_i - 1)
= \sum f_i ({\bf 1}) \lambda_i  = \frac{1}{n} \sum \lambda_i
\;\;\mbox{in $K $}
\end{equation}
and
\begin{equation}
\label{eq1.100}
 \sum f_i (\lambda) = f (\lambda) + \sum f_i (\lambda) (1 - \lambda_i)
\geq f ({\bf 1}) = 1 \;\;\mbox{in $K$}.
\end{equation}

\begin{lemma}
\label{lem1.10}
 Suppose f satisfies (\ref{eq1.40})-(\ref{eq1.80}). Then
\be
 \sum_{i \neq r} f_i \lambda_i^2  \geq 
 \frac1{n-1} (2f |\lambda_r|+ f_r \lambda_r^2 )~\,\, \mbox{if $\lambda_r<0$}
\ee
and so
 \be
 \sum_{i\neq r} f_i \lambda_i^2  \geq \frac1n  \sum f_i \lambda_i^2  ~\,\, \mbox{if $\lambda_r<0$}.
 \ee
\end{lemma}

\begin{proof}
 Suppose $\lambda_r<0$ and order the eigenvalues with $\lambda_1>0$ the largest and  $\lambda_n<0$ the smallest.
Then as a consequence of the  concavity condition (\ref{eq1.50}) we
have
 \be
  \label{eq1.120}
 f_n \geq f_i \,\,\mbox{ for all $i$}
\hspace{.1in} \mbox{and so} \hspace{.1in}
 f_n \lambda_n^2   \geq  f_r \lambda_r^2.
 \ee
 By (\ref{eq1.80}),
 \[ \sum_{i\neq n}f_i \lambda_i = f +  f_n |\lambda_n|. \]
By Schwarz inequality and (\ref{eq1.120}),
\[ f^2 + 2 f f_n |\lambda_n| + f_n^2 \lambda_n^2
   \leq \sum_{i\neq n} f_i \sum_{i \neq n} f_i \lambda_i^2
   \leq (n-1) f_n \sum_{i \neq n}  f_i \lambda_i^2. \]
 Therefore,
\[\sum_{i \neq n}  f_i \lambda_i^2 \geq \frac1{n-1}(2f |\lambda_n|+  f_n \lambda_n^2).\]
Using (\ref{eq1.120}) this implies
 \be
 \label{eq1.130}
 \sum_{i \neq r} f_i \lambda_i^2 \geq  \sum_{i \neq n} f_i \lambda_i^2 \geq
\frac1{n-1}(2f |\lambda_n|+ f_n \lambda_n^2)
 \geq \frac1{n-1}(2f |\lambda_r |+ f_r \lambda_r^2)
\ee completing the proof.
\end{proof}

All of the above assumptions (\ref{eq1.40})-(\ref{eq1.80}) are fairly standard.
In the present work, the following more technical assumption is important.
\begin{equation}
\label{eq1.110}
\lim_{R \rightarrow + \infty}
   f (\lambda_1, \cdots, \lambda_{n-1}, \lambda_n + R)
    \geq 1 + \varepsilon_0 \;\;\;
\mbox{uniformly in $B_{\delta_0} ({\bf 1})$}
\end{equation}
for some fixed $\varepsilon_0 > 0$ and $\delta_0 > 0$,
where $B_{\delta_0} ({\bf 1})$ is the ball
of radius $\delta_0$ centered at ${\bf 1} = (1, \dots, 1) \in \bfR^n$.

The assumption (\ref{eq1.110}) is fairly mild. For $f=H_k^{\frac1k}$
corresponding to the ``higher order mean curvatures'', where $H_k$
is the $k$-th normalized elementary function,
\[\lim_{R\goto \infty} f({\bf 1}+O(\e)+Re_n)=\infty\]
while for  $f=(H_{k,l})^{\frac1{k-l}}=(\frac{H_k}{H_l})^{\frac1{k-l}}~,~k>l$, 
the class of curvature quotients,
\[ \lim_{R\goto \infty}f({\bf 1}+O(\e)+Re_n)
   =(1+O(\e))\Big(\frac{k}{l}\Big)^{\frac1{k-l}}.\]

Problem (\ref{eq1.10})-(\ref{eq1.20}) reduces to the Dirichlet
problem for a fully nonlinear second order equation which we shall
write in the form
\begin{equation}
\label{eq1.140}
G(D^2u, Du, u) = \sigma,
\;\; u > 0 \;\;\; \text{in $\Omega \subset \bfR^n$}
\end{equation}
with the boundary condition
\begin{equation}
\label{eq1.150}
             u = 0 \;\;\;    \text{on $\partial \Omega$}.
\end{equation}
The exact formula of $G$ will be given in Section~\ref{sec2}.

We seek solutions of the Dirichlet problem (\ref{eq1.140})-(\ref{eq1.150})
satisfying $\kappa [u] \equiv \kappa [\mbox{graph} (u)] \in K$. 
Following the literature we define the class of {\em admissible} functions
\[ \mathcal{A} (\Omega) = \{u \in C^2 (\Omega): \kappa [u] \in K\}.
\]


Our main result of the paper may be stated as follows. 

\begin{theorem}
\label{th1.1}
Let $\Gamma = \partial \Omega \times \{0\} \subset \bfR^{n+1}$
where $\Omega$ is a bounded smooth domain in $\bfR^n$ . Suppose that
the Euclidean mean curvature $\mathcal{H}_{\partial \Omega}\geq 0$
and  $\sigma \in (0, 1)$ satisfies
$\sigma>\sigma_0$, where $\sigma_0$ is the unique zero in $(0,1)$ of
\be
 \phi(a):=\frac83 a+\frac{22}{27}a^3-\frac5{27}(a^2+3)^{\frac32}.
 \ee
 (Numerical calculations show $0.3703<\sigma_0<0.3704$.)

Under conditions (\ref{eq1.40})-(\ref{eq1.80}) and  (\ref{eq1.110}),
there exists a complete hypersurface $\Sigma$ in $\bfH^{n+1}$
satisfying (\ref{eq1.10})-(\ref{eq1.20}) with uniformly bounded
principal curvatures
\begin{equation}
 |\kappa [\Sigma]| \leq C \;\; \mbox{on $\Sigma$}.
\end{equation}
Moreover, $\Sigma$ is the graph of a unique admissible solution $u \in C^\infty
(\Omega) \cap C^1 (\bar{\Omega})$ of the Dirichlet problem
(\ref{eq1.140})-(\ref{eq1.150}).
 Furthermore, $u^2 \in C^{\infty} (\Omega)\cap C^{1,1}(\ol{\Omega})$ and
\begin{equation}
\begin{aligned}
&\,\sqrt{1 + |Du|^2} \leq \frac{1}{\sigma}, \;\; u|D^2 u| \leq C
\;\;\; \mbox{in $\Omega$}, \\
&\, \sqrt{1 + |Du|^2} = \frac{1}{\sigma}
\;\;\; \mbox{on $\partial \Omega$}.
\end{aligned}
\end{equation}
\end{theorem}

Theorem~\ref{th1.1} holds for a large family of
$f=\frac1N\sum_{l=1}^N  f_l $ where each $f_l$ consisting of sums
and ``concave products" (that is of the form $(f_1 \cdots
f_{N_l})^\frac1{N_l})$ where each $f_l$  satisfies
(\ref{eq1.40})-(\ref{eq1.80}) and one of the $f_l$ satisfies
(\ref{eq1.110}). 

By \cite{CNS3} condition~(\ref{eq1.40}) implies that equation
(\ref{eq1.140}) is elliptic for admissible solutions. 
As we  shall see in Section~\ref{sec2}, equation (\ref{eq1.140}) is
degenerate where $u = 0$. It is therefore natural to approximate the
boundary condition (\ref{eq1.150}) by
\begin{equation}
\label{eq1.160}
             u = \e > 0 \;\;\;  \text{on $\partial \Omega$}.
\end{equation}
When $\e$ is sufficiently small, the Dirichlet problem
(\ref{eq1.140}),(\ref{eq1.160}) is solvable for all $\sigma \in (0, 1)$.

\begin{theorem}
\label{th1.2}
Let $\Omega$ be a bounded smooth domain in $\bfR^n$ with
$\mathcal{H}_{\partial \Omega}\geq 0$
 and suppose $f$ satisfies (\ref{eq1.40})-(\ref{eq1.80}) and (\ref{eq1.110}).
Then for any $\sigma \in (0, 1)$ and
$\e > 0$ sufficiently small, there exists a unique admissible
solution $u^{\e} \in C^\infty (\bar{\Omega})$ of the Dirichlet
problem (\ref{eq1.140}),(\ref{eq1.160}). Moreover, $u^{\e}$
satisfies the {\em a priori} estimates
\begin{equation} \label{eq.grad}
\sqrt{1 + |D u^{\e}|^2} \leq \frac{1}{\sigma}
\;\;\; \mbox{in $\Omega$}
\end{equation}
\begin{equation}
u^{\e}|D^2 u^{\e}| \leq \frac{C}{\e^2}
\;\;\; \mbox{in $\Omega$}
\end{equation}
where $C$ is independent of $\e$.
\end{theorem}

   We shall use the continuity method to reduce the proof of Theorem~\ref{th1.2} 
to obtaining $C^2$ apriori estimates for admissible solutions. This approach
critically depends on the sharp global gradient estimate (\ref{eq.grad}), which is carried out in
section~\ref{sec3} under the assumption $\mathcal{H}_{\partial \Omega} \geq 0$. It implies that the linearized operator of
equation~\eqref{eq1.140} is invertible for all $\varepsilon \in (0, 1]$, a
crucial condition for the continuity method.
The centerpiece of this paper  is the boundary  second derivative estimate,
which we derive in section~\ref{sec5}.  Here we make use of Lemma~\ref{lem1.10} and a
careful analysis of the linearized operator to derive the mixed
normal-tangential estimate. Again  the sharp global gradient estimate (\ref{eq.grad}) enters into the proof in an essential way.
We then use assumption \eqref{eq1.110} to
establish a pure normal second derivative estimate.  In order to use Theorem \ref{th1.2} to obtain Theorem \ref{th1.1} (see the
end of section 4 for a more detailed explanation), we need a uniform in
$\e$ estimate for the hyperbolic principal curvatures of the $\mbox{graph
$u^{\e}$}$. Therefore in section~\ref{sec6} we prove a maximum principle
for the maximal hyperbolic principal curvature using a method derived in
our earlier paper \cite{GSZ08}.
It is here that we have had to restrict the allowable range of
$\sigma \in (0,1)$.  Otherwise our approach is completely general and we
expect Theorem~\ref{th1.1} is valid for all  $\sigma \in (0,1)$.
In section~\ref{sec2} 
we summarize the basic information about vertical graphs and the
linearized operator that we will need in the sequel and in section \ref{sec3} we review
some important barrier arguments using equidistant sphere solutions.

\bigskip

\section{Vertical graphs and the Linearized operator}
\label{sec2}
\setcounter{equation}{0}

Suppose $\Sigma$ is locally represented as the graph of a function
$u \in C^2 (\Omega)$, $u > 0$, in a domain $\Omega \subset \bfR^n$:
\[ \Sigma = \{(x, u (x)) \in \bfR^{n+1}: \; x \in \Omega\}. \]
oriented  by the upward (Euclidean) unit normal vector
field  $\nu$ to $\Sigma$:
\[ \nu = \Big(\frac{-Du}{w}, \frac{1}{w}\Big), \;\; w=\sqrt{1+|Du|^2}. \]
The Euclidean metric 
and second fundamental form of $\Sigma$ are given respectively by
\[ g^e_{ij} = \delta_{ij} + u_i u_j, \,\, h^e_{ij} = \frac{u_{ij}}{w}. \]
According to \cite{CNS4}, the Euclidean principal curvatures
$\kappa^e [\Sigma]$ are the eigenvalues of the symmetric matrix $A^e
[u] = \{a^e_{ij}\}$:
\begin{equation}
\label{eq2.10}
 a^e_{ij} := \frac{1}{w} \gamma^{ik} u_{kl} \gamma^{lj},
\end{equation}
where
\begin{equation}
\label{eq2.101}
 \gamma^{ij} = \delta_{ij} - \frac{u_i u_j}{w (1 + w)}.
\end{equation}
 Note that the matrix $\{\gamma^{ij}\}$ is invertible with inverse
\begin{equation}
\label{eq2.102}
 \gamma_{ij} = \delta_{ij} + \frac{u_i u_j}{1 + w}
\end{equation}
which is the square root of $\{g^e_{ij}\}$, i.e., $\gamma_{ik}
\gamma_{kj} = g^e_{ij}$. By \eqref{eq1.30} the hyperbolic principal
curvatures $\kappa [u]$ of $\Sigma$ are the eigenvalues of the
matrix $A [u] = \{a_{ij} [u]\}$:
\begin{equation}
\label{eq2.20}
 a_{ij} [u] :=
\frac{1}{w} \Big(\delta_{ij}+ u\gamma^{ik} u_{kl} \gamma^{lj}\Big).
\end{equation}

Let $\mathcal{S}$ be the vector space of $n \times n$ symmetric matrices
and
\[ \mathcal{S}_K = \{A \in \mathcal{S}: \lambda (A) \in K\}, \]
where $\lambda (A) = (\lambda_1, \dots, \lambda_n)$ denotes the
eigenvalues of $A$. Define a function $F$ by
\begin{equation}
\label{eq2.30}
F (A) = f (\lambda (A)), \;\; A \in \mathcal{S}_K.
\end{equation}
Throughout the paper we denote \begin{equation}
\label{eq2.40}
F^{ij} (A) = \frac{\partial F}{\partial a_{ij}} (A), \;\;
  F^{ij, kl} (A) = \frac{\partial^2 F}{\partial a_{ij} \partial a_{kl}} (A).
\end{equation}
The matrix $\{F^{ij} (A)\}$, which is symmetric, has eigenvalues
$f_1, \ldots, f_n$, and therefore  is positive definite for
$A \in \mathcal{S}_K$ if $f$ satisfies (\ref{eq1.40}),
 while (\ref{eq1.50}) implies that $F$ is concave for
$A \in \mathcal{S}_K$ (see \cite{CNS3}), that is
\begin{equation}
\label{eq2.50}
 F^{ij, kl} (A) \xi_{ij} \xi_{kl} \leq 0,
     \;\; \forall \; \{\xi_{ij}\} \in \mathcal{S}, \; A \in \mathcal{S}_K.
\end{equation}
We have
\begin{equation}
\label{eq2.60}
 F^{ij} (A) a_{ij} = \sum f_i (\lambda (A)) \lambda_i,
\end{equation}
\begin{equation}
\label{eq2.70}
F^{ij} (A) a_{ik} a_{jk} = \sum f_i (\lambda (A)) \lambda_i^2.
\end{equation}

The function $G$ in equation (\ref{eq1.140}) is determined by
\begin{equation}
\label{eq2.80}
G (D^2 u, Du, u) = F (A [u])
\end{equation}
where $A [u] = \{a_{ij} [u]\}$ is given by (\ref{eq2.20}).
 Let
  \be \label{eq2.90}
\mathcal{L}=G^{st} \partial_s \partial_t + G^s \partial_s +G_u \ee
be the linearized operator of $G$ at $u$, where
\be \label{eq2.95}
G^{st} = \frac{\partial G}{\partial u_{st}}, \,
 G^s=\frac{\partial G}{\partial u_{s}}, \,
 G_u=\frac{\partial G}{\partial u}. 
 \ee
We shall give the exact formula for $G^s$ later but note that
\begin{equation}
\label{eq2.100}
\begin{aligned}
\,&G^{st}=\frac{u}wF^{ij}\gamma^{is}\gamma^{jt}\\
\,&G^{st}u_{st}=uG_u 
 =F^{ij}a_{ij}-\frac1w \sum F^{ii}
\end{aligned}
\ee
and
\begin{equation}
\label{eq2.110} G^{pq, st} := \frac{\partial^2 G}{\partial u_{pq}
\partial u_{st}} =\frac{u^2}{w^2}  F^{ij,kl} \gamma^{is}\gamma^{tj}
\gamma^{kp}\gamma^{ql}
\end{equation}
where $F^{ij} = F^{ij} (A[u])$, etc. It follows that, under condition (\ref{eq1.40}),
equation~(\ref{eq1.140}) is elliptic
 for $u$ if  $A[u] \in \mathcal{S}_K$, while (\ref{eq1.50})
implies that $G(D^2 u, Du, u)$ is concave with
respect to $D^2 u$.

For later use, the eigenvalues of $\{G^{ij}\}$ and $\{F^{ij}\}$ (which are the $f_i$) are related
by
\begin{lemma}\label{lem2.1}
Let  $0<\mu_1  \leq \ldots  \leq \mu_n$ and $0<f_1 \leq \ldots \leq
f_n$ denote the eigenvalues of $\{G^{ij}\}$ and $\{F^{ij}\}$
respectively. Then \be \label{gs5-F85} w\mu_k \leq u f_k \leq w^3
\mu_k, \;\; 1\leq k \leq n. \ee
\end{lemma}
\begin{proof}
For any $\xi=(\xi_1,\ldots, \xi_n) \in \bfR^n$ we have from
(\ref{eq2.100})
\[u F^{ij}\xi_i \xi_j=wG^{kl}\gamma_{ik}\gamma_{lj}\xi_i \xi_j=wG^{kl}\xi'_k \xi'_l\]
where
\[\xi'_i=\gamma_{ik }\xi_k=\xi_i+ \frac{(\xi \cdot Du)u_i}{1+w}~.\]
Note that
\[ |\xi|^2 \leq |\xi'|^2=|\xi|^2+|\xi \cdot Du|^2 \leq w^2 |\xi|^2 \]
where $\xi'=(\xi'_1, \ldots,\xi'_n)$.
Since both $\{G^{ij}\}$ and $\{F^{ij}\}$  are positive,  (\ref{gs5-F85}) follows from
the minimax characterization of eigenvalues.
\end{proof}

\section{
Height estimates and the asymptotic angle condition} \label{sec3}

In this section let $\Sigma$ be a hypersurface in $\bfH^{n+1}$ with
$\partial \Sigma \subset P(\e):=\{x_{n+1}=\e\}$ so $\Sigma$
separates $\{x_{n+1} \geq \e\}$ into an inside (bounded) region and an
outside (unbounded) one.  Let $\Omega$ be the region in
$\bfR^{n}\times\{0\}$ such that its vertical lift $\Omega^{\e}$ to
$P(\e)$ is bounded by $\partial \Sigma$ (and $\bfR^n \setminus \Omega$
is connected and unbounded). It is allowable that $\Omega$ has several connected
components. Suppose $\kappa [\Sigma] \in K$ and
$f(\kappa)=\sigma \in (0,1)$ with respect to the outer normal.

 Let $B_1=B_R (a)$ be a ball of radius $R$ centered at $a = (a', -\sigma R) \in
\bfR^{n+1}$ where $\sigma \in (0,1)$ and
 $S_1 = \partial B_1 \cap \bfH^{n+1}$. Then $\kappa_i [S_1] = \sigma$ for all
$1 \leq i \leq n$ with respect to its outward normal. Similarly,  let $B_2=B_R (b)$ be a ball of radius $R$ centered at $b = (b', \sigma R) \in
\bfR^{n+1}$ with $S_2 = \partial B_2 \cap \bfH^{n+1}$. Then $\kappa_i [S_2] = \sigma$ for all
$1 \leq i \leq n$ with respect to its inward normal. 

These so called equidistant spheres
 serve as useful barriers. 
 
\begin{lemma}
\label{lem3.0}
\be \label{eq3.05} 
\begin{aligned}
&(i) \,\,\,\, \Sigma \cap \{x_{n+1}<\e\}=\emptyset\\
&(ii) \,\,\mbox{If}\,\, \partial \Sigma \subset B_1, \,\,\mbox{then}\,\, \Sigma \subset B_1~.\\
&(iii)\,\, \mbox{If}\,\,  B_1\cap P(\e)\subset \Omega^{\e}, \,\,\mbox{then}\,\, B_1 \cap \Sigma=\emptyset~.\\
&(iv)\,\,\mbox{If }\,\,B_2 \cap \Omega^{\e}=\emptyset, \,\,\mbox{then}\,\, B_2\cap \Sigma=\emptyset~.
\end{aligned}
\ee
\end{lemma}

\begin{proof} For (i) let $c=\min_{x\in \Sigma}x_{n+1}$ and suppose $0<c<\e$. Then the horosphere
$P(c)$ satisfies $f(\kappa)=1$ with respect to the upward normal, lies below $\Sigma$  and has an
interior contact violating the maximum principle. Thus $c=\e$. For (ii),(iii), (iv) we perform homothetic dilations from $(a^{\prime},0)$ and $(b^{\prime},0)$ respectively
which are hyperbolic isometries and use the maximum principle. For (ii), expand $B_1$ continuously
until it contains $\Sigma$ and then reverse the process. Since the curvatures of $\Sigma$ and $S_1$ are calculated with respect to their outward normals and both hypersurfaces satisfy $f(\kappa)=\sigma$,
there cannot be a first contact. For (iii) and (iv) we shrink
$B_1$ and $B_2$ until they are respectively inside and outside $\Sigma$. When we expand $B_1$
there cannot be a first contact as above. Now shrink $B_2$ until it lies below $P(\e)$ and so is disjoint
(outside) from  $\Sigma$. Now reverse the process and suppose there is a first interior contact. Then the outward normal to $\Sigma$ at this contact point is the inward normal to $S_2$.  Since the curvatures of   $S_2$ are calculated with respect to its inner normal and it satisfies $f(\kappa)=\sigma$, this contradicts the maximum principle.
\end{proof}
\begin{lemma}
\label{lem3.1}
Suppose $f$ satisfies (\ref{eq1.40}), (\ref{eq1.60}) and (\ref{eq1.80}).
 Assume that $\partial \Sigma \in C^2$ and let $u$ denote the height function of $\Sigma$.
Then for $\e > 0$ sufficiently small,
\begin{equation}
\label{eq3.20}
 - \frac{\e \sqrt{1-\sigma^2}}{r_2}
   - \frac{\e^2 (1+\sigma)}{r_2^2}
   < \nu^{n+1}-\sigma
        <  \frac{\e \sqrt{1-\sigma^2}}{r_1} +
        \frac{\e^2 (1-\sigma)}{r_1^2}
\;\;\; \mbox{on $\partial \Sigma$}
\end{equation}
where $r_2$ and $r_1$ are
the maximal radii of exterior and interior spheres to
$\partial \Omega$, respectively.
In particular,  $\nu^{n+1} \rightarrow \sigma$
on $\partial \Sigma$ as $\e \rightarrow 0$.
\end{lemma}
\begin{proof} Assume first $r_2 <\infty$. Fix a point $x_0 \in \partial \Omega$ and let $e_1$ be the outward pointing unit normal to $\partial \Omega$ at $x_0$. Let $B_1, B_2$  be  balls in $R^{n+1}$ with centers
$a_1=(x_0-r_1 e_1, -R_1 \sigma,\, a_2=(x_0+r_2 e_1, R_2 \sigma)$ and radii $R_1, R_2$ respectively
satisfying
\be \label{eq3.30}
R_1^2=r_1^2+(R_1 \sigma+\e)^2,\, R_2^2=r_2^2+(R_2 \sigma-\e)^2~.
\ee
Then $B_1 \cap P(\e)$ is an n-ball of radius $r_1$ internally tangent to $\partial \Omega^{\e}$ at $x_0$ while
$B_2 \cap P(\e)$ is an n-ball of radius $r_2$ externally tangent to $\partial \Omega^{\e}$ at $x_0$ 
By Lemma \ref{lem3.0} (iii) and (iv), $B_i \cap \Sigma=\emptyset, \, i=1,2$. Hence,
\[ -\frac{u-\sigma R_2}{R_2}<\nu^{n+1}<\frac{u+\sigma R_1}{R_1} \,\,\mbox{at $x_0$}~.\]
That is, 
\be \label{eq3.31}
-\frac{\e}{R_2}<\nu^{n+1}-\sigma<\frac{\e}{R_1}\,\,\mbox{at $x_0$}~.
\ee
From (\ref{eq3.30}), 
\[\frac1{R_1}=\frac{\sqrt{(1-\sigma^2)r_1^2+\e^2}-\e \sigma}{r_1^2+\e^2}<\frac{\sqrt{1-\sigma^2}}{r_1}
+\frac{\e(1-\sigma)}{r_1^2}~,\]
and 
\[\frac1{R_2}=\frac{\sqrt{(1-\sigma^2)r_2^2+\e^2}+\e \sigma}{r_2^2+\e^2}<\frac{\sqrt{1-\sigma^2}}{r_2}
+\frac{\e(1+\sigma)}{r_2^2}~,\]
These estimates and \eqref{eq3.31} give \eqref{eq3.20}, completing the proof of the lemma.
\end{proof}

\medskip

\section{The approximating problems and the continuity method}
\label{sec4}

We study the approximating  Dirichlet problem
\begin{equation}
\label{eq4.10}
\begin{aligned}
G(D^2u,Du,u)&=\sigma \hspace{.25in} \mbox{in $\Omega$} \\
u&=\e \hspace{.25in} \mbox{on $\partial \Omega$}
\end{aligned}
\end{equation}
using the continuity method.

Consider for $0\leq t \leq 1$ the family of Dirichlet problems
\begin{equation} \label{eq4.20}
\begin{aligned}
G(D^2 u^t, Du^t, u^t) &=\sigma^t:= t\sigma +(1-t) \hspace{.25in} 
 \mbox{in $ \Omega$,}\\
u^t &= \e \hspace{.25in} \mbox{on $ \partial \Omega$,}\\
u^0 &\equiv \e.
\end{aligned}
\ee

For $\Omega $ a $C^{2+\alpha}$ domain, we find (starting from $u^0
\equiv \e$) a smooth family of solutions $u^t$, $0\leq t \leq 2t_0$
by the implicit function theorem since $G_u|_{u^0} \equiv 0$. We
shall show in a moment that these solutions are unique. By elliptic
regularity it is now well understood that if we can find uniform
estimates in $C^2$  for $0 < t_0 \leq t \leq 1$ then we can solve
(\ref{eq4.10}).

By Lemma~\ref{lem3.0},  we obtain the $C^0$ estimate 
 \be
 \label{eq4.30}
 \e \leq u^t \leq C \,\,\mbox{in $\Omega$}.
 \ee

\subsection{The $C^1$ estimate }

The following proposition shows that we have uniform $C^1$ estimates
in the continuity method and that the linearized operator
$\mathcal{L}$ satisfies the maximum principle.

\begin{proposition}
\label{prop4.1}
 Let $u^t \in C^{2+\alpha}(\overline{\Omega})$ be a family of admissible 
solutions of (\ref{eq4.20}) for $0\leq t \leq t^*$. Suppose
$\mathcal{H}_{\partial \Omega} \geq 0$. Then $G_u|_{u^t} \leq 0$ so
we have uniqueness.
 Hence $w^t$ assumes its maximum on $\partial \Omega$ and
  $w^t\leq \frac{1 }{\sigma^t}$ on $\ol{\Omega}$ for all $0 \leq t \leq t^*$.
\end{proposition}

\begin{proof}
 We (usually) suppress the $t$ dependence for convenience. By (\ref{eq2.100}) and
(\ref{eq1.100}), 
\[u G_u= \sigma^t -\frac{1}{w^t} \sum f_i \leq \sigma^t-\frac1{w^t}. \]
For $t=0,\, \sigma^0=1,~u^0\equiv \e,\,\kappa_i=1,~ f_i=\frac1n$ and
so $uG_u\equiv 0$. Note also that $\frac{d}{dt}(\sigma^t-\frac1{w^t)}|_{t=0}=\sigma-1<0$.
Hence for $t>0$ sufficiently small, $uG_u<0$ so the operator $\mathcal{L}$ given by (\ref{eq2.90})
satisfies the maximum principle. But
$\mathcal{L}u_k=0$  so each derivative $u_k$ achieves
its maximum on $\partial \Omega$. In particular, $w$  assumes its
maximum on $\partial \Omega$. Let $0\in \partial \Omega$ be a
point where $w$ assumes its maximum. Choose coordinates $(x_1,
\ldots , x_n)$ at $0$ with
$x_n$ the inner normal direction
for $\partial \Omega$.  Then at $0$,
 \[ u_{\alpha}=0, \; 1 \leq \alpha < n, \; u_n>0,~ ~ u_{nn}\leq 0,\]
 and
 \[ \sum_{\alpha<n} u_{\alpha \alpha}=-u_n (n-1) \mathcal{H}_{\partial \Omega}\leq 0. \]
Note that by (\ref{eq1.90}), the hyperbolic mean curvature of 
$\mbox{graph $(u)$} \geq \sigma$.
Therefore,
 \[ \frac{n}{\e}\Big(\sigma -\frac1{w}\Big)
  \leq \frac1{w} \Big(\sum_{\alpha<n} u_{\alpha \alpha}+ \frac{u_{nn}}{w^2}\Big)
  \leq -(n-1)\frac{u_n}{w} \mathcal{H}_{\partial \Omega} \leq 0.\]
Hence $\sigma -\frac1w\leq 0$ or $w\leq \frac{1}{\sigma}$.
Thus $G_u \leq 0 $ so $\mathcal{L}$ satisfies the maximum principle.
 Consequently, the same estimates must continue to  hold as we increase $t$ up to $t^*$.
 \end{proof}

In  Section~\ref{sec5}, we will make use of Proposition
\ref{prop4.1} to complete the proof of the $C^2$ estimates (see
Theorem~\ref{th5.1} and Corollary~\ref{cor5.1} ). Since the
linearized operator is invertible, we have unique smooth solvability
all the way to $t=1$ completing the proof of Theorem~\ref{th1.2}. 
Using the global maximum principle, Theorem~\ref{thA1} of
Section~\ref{sec6} and Theorem~\ref{th5.1}, we obtain uniform
estimates for the hyperbolic principal curvatures. Note also that by Lemma \ref{lem3.0} iii, we have
a  positive lower bound (uniform in  $\e$) on each compact subdomain of $\Omega$  for the solutions  $u^{\e}$ obtained in Theorem \ref{th1.2}. This allows us to obtain uniform $C^{2+\alpha}$ estimates for $u^{\e}$ on compact subdomains of $\Omega$ by the interior estimates of Evans-Krylov.  We can now  let $\e$ tend to zero to obtain Theorem~\ref{th1.1}.

\bigskip

\section{Boundary estimates for second derivatives}
\label{sec5}

In this section we establish boundary estimates for second derivatives
of admissible solutions to the Dirichlet problem~\eqref{eq4.20} for all 
$t_0 \leq t \leq 1$. Clearly it suffices to consider the case $t = 1$. 
Throughout this section let
$\Omega$ be a bounded smooth domain in $\bfR^n$ with 
$\mathcal{H}_{\partial \Omega} \geq 0$, 
and $u \in C^3 (\bar{\Omega})$ an admissible solution of the 
Dirichlet problem 
\begin{equation}
\label{eq5.10}
\left\{ \begin{aligned}
G (D^2 u, Du, u) & = \sigma, & \;\;  \mbox{on $\ol{\Omega}$}, \\
u & = \e,  & \;\;  \mbox{on $\partial \Omega$}
\end{aligned} \right.
\end{equation}
where $G$ is defined in (\ref{eq2.80}). 

\begin{theorem}
\label{th5.1}
Suppose that $f$ satisfies (\ref{eq1.40})-(\ref{eq1.80}) and
(\ref{eq1.110}). If $\e$ is sufficiently small,
\begin{equation}
\label{eq5.20}
u|D^2 u| \leq C
\;\;\; \mbox{on $\partial \Omega$}
\end{equation}
where $C$ is independent of
$\e$.
\end{theorem}

 Recall that in section 4, we proved the global gradient estimate $w\leq \frac1{\sigma}$. In particular, 
   $\epsilon \leq u \leq (1+\frac1{\sigma})\epsilon$ in an $\e$ neighborhood of $\partial \Omega$, This
   will be used repeatedly in the proof of Theorem \ref{th5.1} without comment.
   
The notation of this section follows that of Section~\ref{sec2}. Let
$\mathcal{L'}$ denote the partial linearized operator of $G$ at $u$:
\[ \mathcal{L'} = \mathcal{L} - G_u = G^{st} \partial_s \partial_t + G^s \partial_s \]
where $G^{st}, G_u$ are defined in (\ref{eq2.95}) and
\begin{equation} \label{eq5.25}
G^s := \frac{\partial G}{\partial u_s}
    = - \frac{u_s}{w^2 } F^{ij} a_{ij}
      - \frac{2}{w } 
       F^{ij} a_{ik} \Big(\frac{w u_k \gamma^{sj} + u_j \gamma^{ks}}{1+w}\Big)
      +\frac{2}{w^2 } F^{ij} u_i \gamma^{sj}
\end{equation}
by the formula (2.21) in \cite{GS04}, where $F^{ij} = F^{ij} (A
[u])$ and $a_{ij} = a_{ij} [u]$. 

Since $F=\{F^{ij}\}$ and $A=\{a_{ij}\}$  are simultaneously
diagonalizable  by an orthogonal matrix P, we have
\be \label{eq5.26}
|F^{ij} a_{ik}|=(FA)_{jk}=|(P(P^{T}FP)(P^{T}AP)P^{T})_{jk}|=|\sum P_{jr}f_r \kappa_r P_{kr}|\leq 
\sum f_r|\kappa_r|~.
\ee
Hence from (\ref{eq5.25}) and (\ref{eq5.26}),  we obtain,
\begin{lemma}
\label{lem5.1}
Suppose that $f$ satisfies (\ref{eq1.40}), (\ref{eq1.50}), (\ref{eq1.70})
and (\ref{eq1.80}).
Then 
\be
 |G^s| \leq \frac{\sigma}{w} + \frac{2}{w} \sum F^{ii} + 2\sum f_i |\kappa_i|.
 \label{eq5.30} \ee
\end{lemma}

Since $\gamma^{sj} u_s = u_j/w$,
\begin{equation}
G^s u_s = \Big(\frac{1}{w^2} - 1 \Big)  F^{ij} a_{ij} 
          - \frac{2}{w^2} F^{ij} a_{ik} u_k u_j + \frac{2}{w^3} F^{ij} u_i u_j.
\end{equation}
It follows from (\ref{eq2.40}), (\ref{eq2.60}) and (\ref{eq2.100}) that
\begin{equation}
\label{eq5.40}
\mathcal{L'} u = \frac{1}{w^2 } F^{ij} a_{ij} - \frac{1}{w } \sum F^{ii}
         - \frac{2}{w^2 } F^{ij} a_{ik} u_k u_j
         + \frac{2}{w^3 } F^{ij} u_i u_j.
\end{equation}

\begin{lemma}
\label{lem5.5}
Suppose that $f$ satisfies (\ref{eq1.40}), (\ref{eq1.50}), (\ref{eq1.70})
and (\ref{eq1.80}).
Then 
\begin{equation}
\label{eq5.50}
 \mathcal{L} \Big(1 - \frac{\e}{u}\Big)
   \leq - \frac{(1 - \sigma)\e}{ u^2 w} \sum F^{ii} 
        - \frac{2\e}{u^2 w^2 } F^{ij} a_{ik} u_k u_j
\;\; \mbox{in $\Omega$}.
\end{equation}
\end{lemma}
\begin{proof}
By (\ref{eq5.40}), (\ref{eq2.100}) and  (\ref{eq1.80}),
\begin{equation}
\label{eq5.60}
\begin{aligned}
\mathcal{L} \Big(1 - \frac{\e}{u}\Big)
 = & \,\frac{\e}{u^2} \mathcal{L'} u - \frac{2\e}{u^3} G^{st} u_s u_t
           + G_u (1 - \frac{\e}{u}) \\
 = & \, \frac{\e}{u^2 } \Big(\frac{\sigma}{w^2}-\frac1w \sum F^{ii}\Big) 
           + G_u (1 - \frac{\e}{u})
         - \frac{2\e}{u^2 w^2 } F^{ij} a_{ik} u_k u_j.
\end{aligned}
\end{equation}
Since $G_u \leq 0$ by Proposition~\ref{prop4.1}, (\ref{eq5.50}) now follows 
from (\ref{eq1.100}).
\end{proof}

We now refine Lemma \ref{lem5.5}.  For the symmetric matrix $A=A[u]$
we can uniquely define the symmetric matrices (see \cite{RN}) \be
\label{eq5.70}
|A|=\{AA^T\}^{\frac12},\,A^+=\frac12(|A|+A),\,A^-=\frac12(|A|-A) \ee
which all commute and satisfy $A^+A^-=0$. Moreover, $F=\{F^{ij}\}$
commutes with $|A|,\, A^{\pm}$ and so all are simultaneously diagonalizable.
Write $A^{\pm} =
\{a^{\pm}_{ij}\}$ and define
 \be \label{eq5.90}
  L = \mathcal{L}- \frac{2}{w^2 }F^{ij} a^-_{ik} u_k \partial_j. \ee

\begin{lemma}
\label{lem5.10}
Suppose that $f$ satisfies (\ref{eq1.40}), (\ref{eq1.50}), (\ref{eq1.70})
and (\ref{eq1.80}).
Then 
\begin{equation}
\label{eq5.100}
L \Big(1 - \frac{\e}{u}\Big)
   \leq - \frac{(1 - \sigma)\e}{ u^2 w} \sum F^{ii}
\;\; \mbox{in $\Omega$}.
\end{equation}
\end{lemma}

\begin{proof}
Since $\{F^{ij}\}$ is positive definite and simultaneously
diagonalizable with  $A^{\pm}$,
\[ F^{ij} a^{\pm}_{ik} \xi_j \xi_k \geq 0, \;\; \forall \, \xi \in
\bfR^n. \]
Therefore, 
\be \label{eq5.80}
 F^{ij} a_{ik} u_k u_j = F^{ij} (a^+_{ik} - a^-_{ik}) u_k u_j \geq - F^{ij} a^-_{ik} u_k u_j
\ee
 Combining \eqref{eq5.80} and Lemma \ref{lem5.5} we obtain \eqref{eq5.100}.
\end{proof}

The following lemma is stated in \cite{CNS5};  it applies to our situation since horizontal rotations are hyperbolic isometries. For completeness we sketch
the proof.
\begin{lemma}
\label{lem5.15} Suppose that $f$ satisfies (\ref{eq1.40}), (\ref{eq1.50}), (\ref{eq1.70})
and (\ref{eq1.80}).
Then 
\begin{equation}
\label{gsz-B40}
 \mathcal{L} (x_i u_j - x_j u_i) = 0 ,\,\,\mathcal{L} u_i=0, 
\hspace{.25in}1 \leq i, j \leq n.
\end{equation}
\end{lemma}
\begin{proof} Without loss of generality we may assume i=2, j=1. Let
$R(\theta)$ be the orthogonal $n\times n$ matrix with entries $r_{11}=r_{22}=\cos{\theta},\, 
r_{12}=-r_{21}=-\sin{\theta},\,r_{kl}=\delta_{kl}\,\,\mbox{for $3\leq k,l \leq n$}$. Let $y=Rx$ and 
$v(y)=u(x)$. Then since rotations in $x_1, \ldots, x_n$ are hyperbolic isometries, v(y) satisfies
\be \label{eq5.101}
G(D^2v(y), Dv(y), v(y))=\sigma~,
\ee
where 
\be \label{eq5.102}
v(y)=u(R^T y),\, Dv(y)=RDu(R^T y),\, D^2v(y)=R(D^2u(R^T y))R^T~.
\ee
We differentiate (\ref{eq5.101}) with respect to $\theta$ and evaluate at $\theta=0$. With the obvious notation, we obtain
\be \label{eq5.103}
G^{kl}\dot v_{kl}+G^s \dot v_s +G_u \dot v=0
\ee
Using (\ref{eq5.102}) and the definition of R, we compute
\[\dot v=u_i \frac{\partial x_i}{\partial \theta}|_{\theta=0}=u_i \dot r_{pi}(0)x_p=x_2 u_1-x_1 u_2~,\]
\[\dot v_s=\dot r_{si}(0)u_i+r_{si}(0)u_{ij}\dot r_{pj}(0)x_p=x_2 u_{1s}-x_1u_{2s}+u_1\delta_{s2}-u_2\delta_{s1}=(x_2u_1-x_1u_2)_s~,\]
\[\dot v_{kl}=\delta_{ki}\delta_{lj}u_{ijm}\dot r_{nm}(0)x_n+(u_{il}\dot r_{ki}(0)+u_{kj}\dot r_{lj}(0))=
(x_2u_1-x_1u_2)_{kl}~.\]
Hence $\mathcal{L}(\dot v)=0$ as stated. The statement $\mathcal{L}(u_i)=0$ is left to the reader.
\end{proof}

\begin{proof}[Proof of Theorem~\ref{th5.1}]
Consider an arbitrary point on $\partial \Omega$, which we may
assume to be the origin of $\mathbb{R}^n$ and choose the coordinates
so that the positive $x_n$ axis is the interior normal to
$\partial\Omega$ at the origin. There exists a uniform constant
$r > 0$ such that $\partial \Omega \cap B_r (0)$ can be
represented as a graph
\begin{equation}
\label{eq5.110}
 x_n = \rho(x') = \frac{1}{2} \sum_{\alpha,\beta<n}
     B_{\alpha\beta} x_\alpha x_\beta + O (|x'|^3), \qquad
  x' = (x_1, \dots, x_{n-1}).
\end{equation}
Since $u = \e$ on $\partial \Omega$, we see that
$u (x', \rho (x')) = \e$ and
\begin{equation}
\label{eq5.120}
 u_ {\alpha\beta} (0) = - u_n \rho_{\alpha\beta} \qquad \alpha,\beta < n.
\end{equation}
Consequently,
\begin{equation}
\label{eq5.130}
 |u_{\alpha\beta}(0)| \leq C |D u (0)|, \qquad \alpha,\beta < n
\end{equation}
where $C$ depends only on the (Euclidean maximal principal) curvature of
$\partial \Omega$.

As in \cite{CNS1} we consider for fixed $\alpha<n$ the operator
\begin{equation}
\label{eq5.140}
 T_{\alpha} = \partial_\alpha + \sum_{\beta <n}
       B_{\alpha\beta} (x_\beta\partial_n- x_n\partial_\beta).
\end{equation}
Using Lemma \ref{lem5.15} and the boundary condition $u = \e$
on $\partial \Omega$ we have
\begin{equation}
 \label{eq5.150}
\begin{aligned}
 \mathcal{L}T_{\alpha} u = & \, 0, \\ 
   |T_{\alpha} u| + \frac12 \sum_{l<n}u_l^2 \leq & \, C
 \;\; \mbox{in $\Omega \cap B_{\e} (0)$} \\
 |T_{\alpha} u| + \frac12\sum_{l<n}u_l^2 \leq & \, C |x|^2  
 \;\;  \mbox{on $\partial \Omega \cap B_{\e}
 (0)$}.
   \end{aligned}
   \end{equation}

Now define
\[ \phi= \pm T_{\alpha} u + \frac12\sum_{l<n}u_l^2-\frac{C}{\e^2}|x|^2\]
where $C$ is chosen large enough (independent of $\e$)  so that $\phi \leq 0$ 
on $\partial(\Omega \cap B_{\e}(0))$. This is possible by (\ref{eq5.150}).

By (\ref{eq5.30}),  (\ref{eq5.150}),  (\ref{eq2.100}) and Lemma \ref{lem2.1}
\begin{equation}
\label{eq5.160} \mathcal{L}\phi \geq  \sum_{l<n} G^{ij}u_{li}u_{lj}
-\frac{C}{\e} \Big(\sum f_i +\sum f_i |\kappa_i|\Big) \;\; \mbox{in
$\Omega\cap B_{\e}(0)$}.
\end{equation}

Following Ivochkina, Lin  and Trudinger~\cite{ILT} we have

\begin{proposition}
\label{prop5.2}
At each point in $\Omega \cap B_{\e}(0)$ there is an  index $r$ such that
 \be  \label{eq5.170}
  \sum_{l<n} G^{ij}u_{li}u_{lj} \geq c_0 u \sum_{i\neq r}  f_i (\kappa^e_i) ^2 
 \geq \frac{c_0 }{2u}(\sum_{i \neq r} f_i \kappa_i ^2 - \frac2{w^2}\sum f_i)
 \ee
\end{proposition}

\begin{proof}
 Let $P$ be an orthogonal matrix that simultaneously diagonalizes $\{F^{ij}\}$ 
and $A^e$.
By \eqref{eq2.100}  and \eqref{eq2.10}, \be \label{eq5.180}
\begin{aligned}
\sum_{l<n}G^{ij}u_{li}u_{lj}
 = & \, \frac{u}w\sum_{l<n} F^{st} \gamma^{is}\gamma^{jt}u_{li}u_{lj} \\
 = & \, uw\sum_{l<n}F^{st} a^e_{sq} a^e_{pt} \gamma_{pl}\gamma_{ql} \\
 = & \, u w \sum_{l<n} f_i (\kappa^e_i)^2  P_{pi}\gamma_{pl} P_{qi}\gamma_{ql} \\
 = & \, uw\sum_{l<n} f_i (\kappa^e_i)^2 b_{li}^2,
\end{aligned}
\ee
where $B=\{b_{rs}\} = \{P_{ir} \gamma_{is}\}$ 
and $\det{B}=\det(B^T)=w$.

Suppose for some $i$, say $i=1$ that
\[ \sum_{l<n} b_{l1}^2<\theta^2. \]
Expanding $\det B$ by cofactors along the first column gives
\[1\leq w=\det{B}=b_{11}C^{11} 
          +\ldots + b_{ n-1~1}C^{1 n-1}+ b_{n1}\det{M}
\leq c_1 \theta+c_2 \det{M}, \]
where the $C^{1j}$ are cofactors and M is the $n-1$ by $n-1$ matrix
\begin{equation}
M = \begin{bmatrix}
 b_{12}  & \dots   & b_{n-1~2}  \\
 \vdots  & \ddots  & \vdots \\
 b_{1n}  & \ldots  & b_{n-1~n}
   \end{bmatrix}.
\end{equation}
So $\det{M}\geq \frac{1-c_1 \theta}{c_2}$. Now expanding  $\det{M}$
by cofactors along row $r\geq 2$ gives
\[ \det{M}\leq c_3 \left(\sum_{l<n}b_{lr}^2\right)^{\frac{1}{2}} \]
by Schwarz inequality. Hence
 \be \label{eq5.190}
 \sum_{l<n}b_{lr}^2 \geq \Big(\frac{1-c_1 \theta}{c_2 c_3}\Big)^2.
 \ee
 Choosing $\theta < \frac1{2c_1}$, \eqref{eq5.190} and \eqref{eq5.180} imply
\[\sum_{l<n}G^{ij}u_{li}u_{lj} \geq c_0 u \sum_{i\neq r} f_i
(\kappa^e_i)^2 \;\; \mbox{for some $r$}.\]
 Finally using
$\kappa^e_i=\frac1u(\kappa_i-\frac1w)$, \eqref{eq5.170} follows.
\end{proof}

\begin{proposition}
 \label{prop5.3}
 Let L be defined by \eqref{eq5.90}. Then
\[ L \phi \geq -C_1 \Big(G^{ij}\phi_i \phi_j +\frac1{\e}\sum f_i\Big) \]
 for a controlled  constant $C_1$ independent of $\e$.
\end{proposition}

\begin{proof}
By the generalized Schwarz inequality ,
 \be \label{eq5.195}
\begin{aligned}
\frac2{w^2} |F^{ij} a^-_{jk} u_i \phi_k|
 \leq & \, 2\Big(uF^{ij} \phi_i \phi_j\Big)^{\frac12} \Big(\frac1u F^{ij} a^-_{il} a^-_{kj} \frac{u_k u_l}{w^2}\Big)^{\frac12} \\
 \leq & \, \frac{c_0}{8nu}\sum_{\kappa_i<0} f_i \kappa_i^2 + CG^{ij}\phi_i \phi_j
\end{aligned}
\ee where we have used Lemma \ref{lem2.1} to compare $uF^{ij}\phi_i
\phi_j$ to $G^{ij}\phi_i \phi_j$.

Since (recall (\ref{eq1.80}))
\[ \sum f_i |\kappa_i| =\sigma +2\sum_{\kappa_i<0} f_i |\kappa_i|, \]
using  \eqref{eq5.195}, \eqref{eq5.160}, Proposition~\ref{prop5.2}
and Lemma \ref{lem1.10} we have
 \be \label{eq5.15}
\begin{aligned}
 L \phi & \geq  \frac{c_0}{2u} \sum_{i\neq r} f_i \kappa_i^2
               - \frac{c_0}{4nu}\sum_{\kappa_i<0} f_i \kappa_i^2
               - C \Big(G^{ij}\phi_i \phi_j +  \frac{1}{\e}\sum f_i\Big)\\
        & \geq - C_1 \Big(G^{ij}\phi_i \phi_j +  \frac{1}{\e}\sum f_i\Big)
\end{aligned}
\ee
for a controlled constant $C_1$ independent of $\e$.
\end{proof}

Let $h=(e^{C_1\phi}-1)-A(1-\frac{\e}u)$ with $C_1$ the constant in
Proposition~\ref{prop5.3} and $A$ large compared to $C_1$. By
Proposition~\ref{prop5.3} and Lemma~\ref{lem5.10},
\[h\leq 0 \;\; \mbox{on $\partial(\Omega \cap B_\e (0))$}\]
and
\[Lh \geq 0 \;\; \mbox{in $\Omega\cap B_{\e}(0)$}. \]
By the maximum principle $h \leq 0$ in $\Omega \cap B_{\e}(0)$.
Since $h (0) = 0$, we have $h_n (0) \leq 0$ which gives
\begin{equation}
\label{eq5.200}  |u_{\alpha n} (0)|\leq \frac{A }{C_1 \e} u_n(0).
\end{equation}

Finally, $|u_{nn} (0)|$ can be estimated as in \cite{GSZ08}
using hypothesis (\ref{eq1.110}). For completeness we include the argument 
here. We may assume $[u_{\alpha \beta} (0)]$, $1 \leq \alpha, \beta < n$, to be
diagonal. Note that $u_{\alpha} (0) = 0$ for $\alpha < n$.
We have at $x = 0$
\[ A [u] = \frac{1}{w} \begin{bmatrix}
   1 + u u_{11} &           0  & \dots   & \frac{u u_{1n}}{w} \\
            0   & 1 + u u_{22} & \dots   & \frac{u u_{2n}}{w} \\
   \vdots       &   \vdots     & \ddots  & \vdots \\
\frac{u u_{n1}}{w} & \frac{u u_{n2}}{w} & \dots  & 1 + \frac{u u_{nn}}{w^2}
   \end{bmatrix}. \]

By Lemma 1.2 in \cite{CNS3}, if $\varepsilon u_{nn}(0)$ is very
large, the eigenvalues $\lambda_1, \ldots , \lambda_n$ of
$A [u]$ are asymptotically given by
\begin{equation}
\begin{aligned}
 \lambda_{\alpha} = & \frac{1}{w} (1 + \e u_{\alpha \alpha} (0)) + o(1),
\; \alpha < n \\
       \lambda_n  = & \frac{\e u_{nn}(0)}{w^3}
             \Big(1 + O \Big(\frac{1}{\e  u_{nn} (0)}\Big)\Big).
\end{aligned}
\end{equation}
By (\ref{eq5.130}) and assumptions (\ref{eq1.80}),(\ref{eq1.110}),
for all $\e > 0$ sufficiently small,
\[ \sigma = \frac{1}{w} F (w A[u] (0))
   \geq  \frac{1}{w} \Big(1 + \frac{\varepsilon_0}{2}\Big)    \]
if $\e u_{nn}(0) \geq R$ where $R$ is a uniform constant.
By the hypothesis (\ref{eq1.110}) and Proposition~\ref{prop4.1} however,
\[ \sigma \geq \frac{1}{w} \Big(1 + \frac{\varepsilon_0}{2}\Big)
     \geq \sigma
           \Big(1 + \frac{\varepsilon_0}{2}\Big)
       > \sigma \]
which is a contradiction.
Therefore 
\[ \e |u_{nn} (0)| \leq  R\]
and the proof is complete.
\end{proof}

Applying the maximum principle for the largest principal curvature
$\kappa_{\max}$  obtained in Theorem 5.2 of \cite{GSZ08} we obtain

\begin{corollary} \label{cor5.1}
 Let $\Omega$ be a bounded smooth domain in $\bfR^n$ with 
$\mathcal{H}_{\partial \Omega} \geq 0$, and
$u \in C^3 (\bar{\Omega}) \cap C^4 (\Omega)$ an admissible solution
of problem (\ref{eq5.10}).
Suppose that $f$ satisfies (\ref{eq1.40})-(\ref{eq1.80}) and  (\ref{eq1.110}).
Then, if $\e$ is sufficiently small,
\begin{equation}
u |D^2 u| \leq \frac{C}{\e^2}
\;\;\; \mbox{in $\ol{\Omega}$}
\end{equation}
where $C$ is independent of $\varepsilon$.
\end{corollary}

Note that Corollary~\ref{cor5.1} suffices to complete the proof of
Theorem~\ref{th1.2} but that we cannot use it to pass to the limit
as $\e$ tends to zero. In the following section we address this
problem.

\bigskip

\section{Global estimates for second derivatives}
\label{sec6}

In this section we prove a maximum principle for the largest
hyperbolic principal curvature $\kappa_{\max} (x)$ of solutions of
$f(\kappa[u])=\sigma$. We make extensive use of our previous
calculations in Section 5 of \cite{GSZ08}. 

Let $\Sigma$ be the graph of $u$. For $x \in \Omega$ let $\kappa_{\max} (x)$ 
be the largest principal curvature of $\Sigma$ at the point 
$X = (x, u (x)) \in \Sigma$.  We define as in \cite{CNS5},
\[ M_0 = \max_{x \in \ol{\Omega}} \frac{\kappa_{\max} (x)}{\eta-a},\]
 where $\eta =  \nu^{n+1}=\ve \cdot \nu$ is the upward (Euclidean) unit 
normal to $\Sigma$ and $\inf \eta \geq \sigma>a$.

\begin{theorem}\label{thA1}
 Suppose that $f$ satisfies (\ref{eq1.40})-(\ref{eq1.80}) and
$\sigma \in (0,1)$ satisfies  $\sigma>\sigma_0$, where $\sigma_0$ is
the unique zero  in $(0,1)$ of
 \be \label{eqA10}
 \phi(a):=\frac83 a+\frac{22}{27}a^3-\frac5{27}(a^2+3)^{\frac32}.
 \ee
Let $u \in C^4 (\Omega)$ be an admissible solution of \eqref{eq5.10}
such that 
$\nu^{n+1}=\frac1w \geq \sigma$. Then  at an interior maximum of
$M_0$,
\[ \kappa_{\max} \leq \frac{C}{\sigma-\sigma_0} \]
 where $C$ is independent of $\e$. 
Numerical calculations show $0.3703<\sigma_0<0.3704$.
\end{theorem}

\begin{proof}
Suppose $M_0$ is attained at an interior point $x_0 \in \Omega$
and let  $X_0 = (x_0, u (x_0))$.
After a horizontal translation of the origin in $\bfR^{n+1}$,
 we may write $\Sigma$ locally near $X_0$ as a radial graph
\begin{equation}
\label{gsz-F200'}
X = e^v \vz, \;\;\; {\vz} \in \bfS^n_+ \subset \bfR^{n+1}
\end{equation}
with $X_0 = e^{v (\vz_0)} \vz_0$, $\vz_0 \in \bfS^n_+$, such that
$\nu (X_0) = \vz_0$. Note that the height function $u = y e^v$,
and upward unit (Euclidean) normal is $\nu=\frac{\vz-\nabla v}w$
where $y=\ve \cdot \vz$ and $ w = (1 + |\nabla v|^2)^{\frac{1}{2}}$
. (Here $\ve$ is the vertical unit vector
pointing upwards.) Hence $\eta=\frac{y-\ve \cdot \nabla v}w$.

We choose an orthonormal local frame $\tau_1, \ldots, \tau_n$ around
$\vz_0$ on $\bfS^n_+$ such that 
$v_{ij} = \nabla_{\tau_j} \nabla_{\tau_i} v$ is diagonal at $\vz_0$,
where $\nabla$ denotes the Levi-Civita connection on $\bfS^n$. 
As shown in Section 2.2 of \cite{GSZ08}, the
hyperbolic principal curvatures of the radial graph $X$ are the
eigenvalues of the matrix $A^{\vs} [v] = \{a^{\vs}_{ij}[v]\}$:
 \begin{equation}
\label{gsz-F20'}
 a^{\vs}_{ij}[v] := \frac{1}{w} (y \gamma^{ik} v_{kl} \gamma^{lj}
               - \ve \cdot \nabla v \delta_{ij})
\end{equation}
where 
\[ \gamma^{ij} = \delta_{ij} - \frac{v_i v_j}{w (1 + w)}\,
 . \]
We can then rewrite equation~\eqref{eq5.10} in the form
\begin{equation}
\label{gsz-F100'} F (A^{\vs} [v]) = \sigma.
\end{equation}
Henceforth we write $A [v] = A^{\vs} [v]$ and $a_{ij} = a^{\vs}_{ij} [v]$.

 Since $\nu (X_0) = \vz_0$,  $\nabla v (\vz_0) = 0$ and hence
\begin{equation}
\label{gsz-G55}
 a_{ij} = y v_{ij} = \kappa_i
\delta_{ij}
\end{equation}
at $\vz_0$ by (\ref{gsz-F20'}), where $\kappa_1, \ldots, \kappa_n$ 
are the principal curvatures of $\Sigma$ at $X_0$. 

We note that at $\vz_0$
\be \label{eq6.10}
\begin{aligned}
&y_i=\nabla_i( \ve\cdot \vz)=e\cdot \nabla_i \vz=\ve\cdot \tau_i\\
&(\ve\cdot \nabla v)_k=\ve\cdot v_{ik}\tau_i=y_iv_{ik}=y_k v_{kk}\\
&\eta_i=(\frac{y-e\cdot \nabla v}w)_i= y_i(1-v_{ii})\\
&a_{ij,k}=yv_{ijk}+y_k (v_{ii}-v_{kk}) \delta_{ij}\\
&v_{ijk}=v_{ikj}=v_{kij}\\
&y(a_{i1,1}-a_{11,i})=y_i(\kappa_i-\kappa_1)
\end{aligned}
\ee

We may assume
\begin{equation}
\label{gsz-G60}
 \kappa_1 = \kappa_{\max} (X_0).
\end{equation}
The function $\frac{a_{11}}{\eta -a}$, which is defined locally near
$\vz_0$, then achieves its maximum at $\vz_0$.  Therefore at  $\vz_0$,
\begin{equation}
\label{gsz-G70}
 \Big(\frac{a_{11}}{\eta-a}\Big)_i = 0, \;\; 1 \leq i \leq n
\end{equation}
 and
\begin{equation}
\label{gsz-G80}
y^2(y-a) F^{ii} a_{11, ii} - y^2 \kappa_1 F^{ii} \eta_{ii} 
=y^2(y-a)^2F^{ii} \Big(\frac{a_{11}}{\eta-a}\Big)_{ii}  \leq 0.
\end{equation}

The left hand side of (\ref{gsz-G80}) is exactly calculated (these calculations are long) in
 Proposition~5.3
and Lemma~5.4 of \cite{GSZ08}(with $\phi=\eta$) and yield
\begin{equation}
\label{gsz-G500}
\begin{aligned}
\sigma (y-a)\kappa_1^2  + & \, a\kappa_1 \sum f_i \kappa_i^2 + (a - 2 (1-y^2) (y - a))  \kappa_1 \sum f_i & \\
  \leq & \, 2\sigma \kappa_1 + \frac{2 a\kappa_1}{\alpha} \sum  f_i (\kappa_i- \alpha) \,y_i^2   \\
       & \, - \frac{2 a^2 \kappa_1^2}{\alpha^2 (y-a)} \sum_{i=2}^n
     \frac{f_i -f_1}{\kappa_1-\kappa_i} (\kappa_i- \alpha)^2 \,y_i^2
  \end{aligned}
\end{equation}
where $\alpha=\frac{a\kappa_1}{\kappa_1-(y-a)}$.
We note only that differentiation of the equation (\ref{gsz-F100'}) twice gives
\be
y^2(y-a)F^{ii}a_{ii,11}=-y^2(y-a)F^{ij,rs}a_{ij,1}a_{kl,1}
\ee
 and the last term in \eqref{gsz-G500} comes from this ``concavity
term''
\begin{equation}
\label{gsz-G300}
 -y^2(y-a) F^{ij,kl} a_{ij,1} a_{kl,1} \geq
 2(y-a)\sum_{i=2}^n \frac{f_i - f_1}{\kappa_1 - \kappa_i} (ya_{i1,1})^2
   \end{equation}
 where, since $(\frac{a_{11}}{\eta -a})_i = 0$, that is $a_{11,i}=\frac{\kappa_1}{y-a}\eta_i$, 
 we find using (\ref{eq6.10})
 \[ ya_{i1,1}= ya_{11,i} + (\kappa_i - \kappa_1) y_i=(a\kappa_1-(\kappa_1-(y-a))\kappa_i)\frac{y_i}{y-a}
            = -\frac{a \kappa_1 (\kappa_i- \alpha) y_i}{\alpha  (y-a)}. \]
 We also recall  the identity
 \[ \sum y_i^2 = 1-y^2\]
which has been used in \eqref{gsz-G500}, which follows from
\[y_i=\nabla_i(\ve \cdot \vz)=\ve\cdot \tau_i \,\mbox{and}\, \ve=\sum(\ve \cdot \tau_i)\tau_i + 
y \vz~. \]

It was shown in Section 6 of \cite{GSZ08} that the coefficient
$\gamma(y)$ of $\kappa_1 \sum f_i$ in \eqref{gsz-G500}, namely
   \begin{equation}
 \label{gsz-G550}
 \gamma(y)=a - 2 (1-y^2) (y - a) ~\,\,
\end{equation}
  satisfies
  \be \label{eqA50}
  \gamma(y)>\frac73 a -\frac4{27}a^3-\frac4{27}(a^2+3)^{\frac32}>0 
\hspace{.1in}\mbox{on $(a,1)$}
  \ee
  if $a^2>\frac18$. Therefore the terms on the left hand side of \eqref{gsz-G500} are all positive and we have one term of order $\kappa_1^2$. The only ``dangerous'' term on the right hand side of  \eqref{gsz-G500}  is the second one and we may throw away those terms in that sum where $\kappa_i \leq \alpha$. Thus we need only concern 
ourselves with
\[I=\{i: \kappa_i > \alpha > a\}. \]

Fix $\theta \in (0,1)$ to be chosen  later and let
\[ \begin{aligned}
 J = & \, \{i\in I: f_1 \leq \theta f_i\}, \\
 K = & \, \{i\in I: f_1>\theta f_i\}.
 \end{aligned} \]
Then
 \be \label{eqA80}
  a \kappa_1 \sum_{i\in J} f_i \kappa_i^2 > a^3
\kappa_1 \sum_{i\in J} f_i
 \ee
 and
 \be \label{eqA60}
  \frac{2 a\kappa_1}{\alpha} \sum_{i\in K} f_i (\kappa_i- \alpha) \,y_i^2
- a \kappa_1^3 f_1
\leq \kappa_1^2 (\frac2{\theta}-a\kappa_1) f_1 <0,
 \ee
 provided $\kappa_1> \frac2{a\theta}$. On the other hand  by Cauchy-Schwarz (or completing the 
 square),
  \be \label{eqA70}
\begin{aligned}
 \sum_{i\in J} & \, f_i (\kappa_i - \alpha) \,y_i^2
        - \frac{a \kappa_1}{\alpha (y-a)} \sum_{i\in J}
 \frac{f_i -f_1}{\kappa_1-\kappa_i} (\kappa_i- \alpha)^2 \,y_i^2 \\
 \,& \leq \sum_{i\in J} f_i y_i^2 \Big((\kappa_i - \alpha)  
     - \frac{(1-\theta)a}{\alpha(y-a)} (\kappa_i - \alpha)^2\Big)\\
 \,& \leq \frac{\alpha (y-a)(1-y^2)}{4 (1-\theta) a} \sum_{i\in J} f_i \\
 \, &  =  \frac{\alpha (a-\gamma(y))}{8 (1-\theta) a} \sum_{i \in J}f_i.
 \end{aligned}
 \ee
Combining \eqref{gsz-G500}, \eqref{eqA80}, \eqref{eqA60} and
\eqref{eqA70} we obtain
\begin{equation}
\label{gsz-G500'}
 \sigma (y-a)\kappa_1^2  + \phi_{\theta} (y)
 \kappa_1 \sum_{i \in J} f_i  \leq 2\sigma \kappa_1
\end{equation}
where the coefficient of $\kappa_1 \sum_{i\in J} f_i $ in
\eqref{gsz-G500'} is
\[ \phi_{\theta}(y)=\gamma(y)- \frac{a-\gamma(y)}{4(1-\theta)}+a^3.\]
Note that by \eqref{eqA50},
\be \label{eqA90}
\begin{aligned}
\,&\phi_0(y)=\frac54\{\gamma(y)+\frac45 a^3-\frac{a}5\} \\
\,&>\frac54\{\frac73a-\frac4{27}a^3-\frac4{27}
(a^2+3)^{\frac32}+\frac45 a^3-\frac{a}5 \} \\
\,&= \frac83 a+\frac{22}{27}a^3-\frac5{27}(a^2+3)^{\frac32}=: \phi(a).
\end{aligned}
\ee

For $a\in (0,1)$ it is easily checked that $\phi'(a)>0,\,
\phi(0)<0,\,\phi(1)>0$. Let $\sigma_0$ be the unique zero of
$\phi(a)$ in $(0,1)$. Numerical calculations show that
$0.3703<\sigma_0<0.3704$.

Now assume that $2\e_0:=\sigma-\sigma_0>0$ and choose
$a=\sigma_0+\e_0$. Then $\phi_{\theta}(y)>0$ on $(a,1)$ if
$\theta>0$ is chosen sufficiently small.  By Proposition
\ref{prop4.1}, $y-a\geq \sigma-a\geq \e_0$ at $\vz_0$, so by
\eqref{gsz-G500'} (assuming $\kappa_1> \frac2{a\theta}$) we obtain
$\e_0 \kappa_1^2 \leq 2\kappa_1 $. Hence
\[\kappa_1 \leq 2 \max \Big\{\frac{1}{a\theta}, \frac{1}{\e_0}\Big\}
   = 4 \max \Big\{\frac{1}{\theta (\sigma + \sigma_0)}, \frac{1}{\sigma - \sigma_0}\Big\}   \]
and so (since $\eta-a<1$)
\[\max_{x \in \ol{\Omega}} \kappa_{\max} (x)\leq \frac{ \kappa_1(\vz_0)}{\e_0}\leq 8 \max \Big\{\frac{1}{\theta (\sigma^2 - \sigma_0^2)}, \frac{1}{(\sigma - \sigma_0)^2}\Big\}   \]
completing the proof of Theorem \ref{thA1}.
\end{proof}

\end{document}